\documentclass{article}
\usepackage{spconf,amsmath,amssymb,amsthm,graphicx,booktabs,array}
\usepackage{algorithm}
\usepackage{algpseudocode}

\usepackage{url}
\usepackage{iftex}
\ifPDFTeX\else
  \usepackage{fontspec}
\fi

\newcommand{\SPD}{\mathbb{S}_{++}^{n}}
\newcommand{\Sym}{\mathbb{S}^{n}}
\newcommand{\R}{\mathbb{R}}
\newcommand{\tr}{\operatorname{tr}}
\newcommand{\grad}{\operatorname{grad}}
\newcommand{\Hess}{\operatorname{Hess}}
\newcommand{\dist}{\operatorname{dist}}
\newcommand{\frob}[1]{\lVert #1\rVert_{\mathrm{F}}}
\newcommand{\norm}[1]{\lVert #1\rVert}

\newcommand{\lcp}{p}
\newcommand{\lcq}{q}
\newcommand{\lcr}{r}
\newtheorem{proposition}{Proposition}
\newtheorem{corollary}[proposition]{Corollary}

\makeatletter
\long\def\@makecaption#1#2{%
  \vskip 10pt
  \setbox\@tempboxa\hbox{#1. #2}%
  \ifdim \wd\@tempboxa >\hsize #1. #2\par \else
    \hbox to\hsize{\hfil\box\@tempboxa\hfil}%
  \fi
  \vskip\belowcaptionskip}
\makeatother
\title{OPTIMIZATION OVER COVARIANCE MATRICES WITH A PARAMETERIZED METRIC}

\name{Yibang Li$^{1}$, Bamdev Mishra$^{2}$, Pratik Jawanpuria$^{3}$ and
Cyrus Mostajeran$^{1}$}
\address{$^{1}$Nanyang Technological University, Singapore \\
  $^{2}$Microsoft India \qquad
  $^{3}$Indian Institute of Technology Bombay, India \\
  yibang001@e.ntu.edu.sg \quad bamdevm@microsoft.com \\
  pratik.jawanpuria@iitb.ac.in \quad cyrussam.mostajeran@ntu.edu.sg}

\begin{document}
\ninept
\maketitle

\begin{abstract}
The choice of Riemannian metric can strongly influence the convergence of gradient-based optimization over covariance matrices. Euclidean, Bures--Wasserstein and affine-invariant metrics are common choices, but their relative effectiveness depends on the objective.
We introduce a two-parameter family defined
by $X^{p}LX^{q}+X^{q}LX^{p}=U$, solved for $L$ at each tangent vector $U$, that
contains all three as exact members, at $(0,0)$, $(1,0)$ and $(1,1)$, and
extends past them. We treat the choice of member as a particular way of
preconditioning for a given problem. To this end, we analyze the conditioning of
the Riemannian Hessian at the solution. We show that it obeys a lower bound that
depends on $(p,q)$ only through the exponent $r=p+q$. When the Euclidean Hessian
is a pure power that mixes no eigendirections, the member $p=q=r/2$ attains that
bound, and a closed-form criterion identifies the other members that do. We
discuss ways to tune $r$ for a given problem. Experiments on real covariance data
confirm the predicted
conditioning and the benefit of tuning $r$. A task covariance example shows a further gain from tuning the shape. 
\end{abstract}

\begin{keywords}
Covariance matrices, Riemannian optimization, preconditioning,
Bures--Wasserstein, metric learning
\end{keywords}

\section{INTRODUCTION}

Symmetric positive-definite (SPD) matrices are the working object in covariance
estimation, adaptive beamforming, radar detection, diffusion tensor imaging
\cite{pennec2006riemannian}, and brain--computer interfaces
\cite{barachant2012multiclass}. The set of SPD matrices forms a manifold $\SPD$,
an open subset of the space $\Sym$ of symmetric matrices, so $\Sym$ is its
tangent space at every point. Endowed with a metric, it has the structure of a
Riemannian manifold. Minimizing a function $f$ over $\SPD$ by gradient descent 
requires a choice of Riemannian metric, which determines the gradient and hence the descent direction.

Each metric in common use on the set of SPD matrices is often
argued for on geometric grounds. Affine-invariant geometry is complete and
congruence invariant. Bures--Wasserstein geometry is
the covariance part of quadratic optimal transport between Gaussian measures
\cite{takatsu2011wasserstein,malago2018wasserstein,bhatia2019bures}.
Log-Euclidean geometry is widely used in imaging
\cite{arsigny2007geometric}. Thanwerdas and Pennec
\cite{thanwerdas2021invariant} classify the $O(n)$-invariant metrics and show
that the \emph{kernel metrics} of Hiai and Petz \cite{hiai2009means}, each fixed
by a single function of two eigenvalues, form a subclass containing most of the
metrics in common use, and elsewhere they build continua that interpolate
between named metrics \cite{thanwerdas2019affine,thanwerdas2021mixed}.

Since
the metric
is what converts $\nabla f$ into a search direction, it acts as an intrinsic
preconditioner rather than a passive modeling choice
\cite{mishra2016preconditioning}. Picking one is therefore an algorithmic
decision rather than a geometric one. This is, for example, explored by Han et
al.\ \cite{han2021riemannian} who compare Bures--Wasserstein (BW) and affine-invariant
geometry and report that each wins on a different objective class,
while \cite{han2023generalized} introduce a parameterized generalization. More
concretely, preconditioning on manifolds has attracted much attention.
Mishra and Sepulchre \cite{mishra2016preconditioning} build the metric
from the Hessian of the Lagrangian. Shustin and Avron
\cite{shustin2023preconditioning} choose the metric on the generalized Stiefel
manifold through a preconditioning scheme, motivate the choice by the condition
number of the Riemannian Hessian at the optimum, and identify the ideal
preconditioner as the Euclidean Hessian there. Gao et al.\
\cite{gao2025product} do the same on product manifolds. Closest to this work,
Zhou et al.\ \cite{zhou2026alpha} discuss Alpha-Procrustes metrics (which
include log-Euclidean and BW) for optimization, especially from a robustness
viewpoint.

\textbf{Contributions.} We explore the question of metric choice for the SPD
manifold. To this end, our contributions are the following. We introduce a two-parameter
metric family containing Euclidean, Bures--Wasserstein and affine-invariant
geometry as exact members. We show how the condition number of the Hessian depends on $p$ and $q$, and motivate ways to approximate this. Finally, our experiments show the benefit of tuning the metric.

\section{A PARAMETERIZED METRIC}
\label{sec:family}

\textbf{The metric.}
Fix two real numbers $p$ and $q$, and let $X\in\SPD$ be the point at which the
metric is being defined. This follows the usual definition
of Bures--Wasserstein geometry \cite{bhatia2019bures}. For a tangent vector
$U\in\Sym$ let $L=L_{p,q}(U)$ solve the Sylvester-type equation
\begin{equation}
X^{p}LX^{q}+X^{q}LX^{p}=U.
\label{eq:lyap}
\end{equation}
For tangent vectors $U$ and $V$, the Riemannian metric $g$ is defined as
\begin{equation}
g^{(p,q)}_X(U,V)=\tfrac{1}{2}c_{p,q}\,\tr\bigl(L_{p,q}(U)\,V\bigr),\qquad
c_{p,q}=4^{\,1-(p-q)^{2}},
\label{eq:metric}
\end{equation}
where $c_{p,q}$ is a normalizing constant. Let $d_1,\dots,d_n$ and $P$ hold the
eigenvalues and eigenvectors of $X$, and write $M'=P^{\top}MP$ for the
eigenbasis coordinates of any $M\in\Sym$. In this eigenbasis, the operator on the left-hand side of \eqref{eq:lyap} acts entrywise, multiplying $L'_{ij}$ by $d_i^p d_j^q+d_i^q d_j^p$. Since the eigenvalues of $X$ are positive, these coefficients are strictly positive for all real $p$ and $q$. The operator is therefore self-adjoint, positive definite, and invertible, so \eqref{eq:metric} defines a Riemannian metric for every such pair. At $(p,q)=(1,0)$, \eqref{eq:lyap} reduces to $XL+LX=U$, and \eqref{eq:metric} recovers the Bures--Wasserstein metric. 

\textbf{The weight.}
Solving \eqref{eq:lyap} entry by entry puts \eqref{eq:metric} in the kernel form
of Hiai and Petz \cite{hiai2009means},
\begin{equation}
g_X^{\phi}(U,V)=\sum_{i,j}\frac{U'_{ij}V'_{ij}}{\phi(d_i,d_j)},\qquad
\phi=\phi_{p,q},
\label{eq:kernel}
\end{equation}
with the \emph{weight} in the denominator,
\begin{equation}
\phi_{p,q}(x,y)=\tfrac12\,4^{\,(p-q)^{2}}\bigl(x^{p}y^{q}+x^{q}y^{p}\bigr),
\label{eq:pq}
\end{equation}
symmetric in $(p,q)$ and positively homogeneous of degree $r=p+q$, meaning
$\phi_{p,q}(tx,ty)=t^{\,r}\phi_{p,q}(x,y)$ for every $t>0$. We call $r$
the \emph{exponent} of the member of the metric family.

\textbf{Named members.}
Evaluating \eqref{eq:pq} at $(p,q)=(0,0)$, $(1,0)$ and $(1,1)$ gives the weights
$1$, $2(x+y)$ and $xy$, of degrees $0$, $1$ and $2$. These are exactly the
Euclidean,
Bures--Wasserstein and affine-invariant metrics, respectively. The metric family interpolates between them, and since $p$ and $q$ may be any
reals, it extends past them in every direction. On the diagonal $p=q=r/2$
\eqref{eq:lyap} reads $2X^{r/2}LX^{r/2}=U$, so
$L=\tfrac12X^{-r/2}UX^{-r/2}$ and \eqref{eq:metric} becomes the power family in
closed form,
\begin{equation}
g_X^{\phi}(U,V)=\tr\!\bigl(X^{-r/2}UX^{-r/2}V\bigr),\qquad
\phi(x,y)=(xy)^{r/2}.
\label{eq:slice}
\end{equation}

\textbf{Relation to known metric families.}
Every smooth positive $\phi$ defines a kernel metric
\cite{hiai2009means}, so \eqref{eq:pq} is a subfamily of a known class, singled
out by the factorization
\begin{equation}\label{eq:known_family_metric}
\phi_{p,q}(x,y)=4^{\,(p-q)^{2}}(xy)^{r/2}
\cosh\!\bigl(\tfrac{p-q}{2}\log(x/y)\bigr),
\end{equation}
in which $r$ fixes the homogeneity degree and $p-q$ the shape at fixed degree.
That separation is what makes the conditioning law of Section~\ref{sec:degree}
possible. It is the separation rather than any individual member that is new. The family also shares members with the mixed-power-Euclidean family of \cite{thanwerdas2021mixed}. The inverse-Euclidean metric is the pullback of the Euclidean metric under $X\mapsto X^{-1}$. It is our diagonal member at $r=4$ and the mixed-power-Euclidean member at $(-1,-1)$. Off the diagonal, our member $(p,q)=(\tfrac12,0)$ agrees up to a constant factor with their member at $(1,\tfrac12)$. The family
differs from the generalized Bures--Wasserstein geometry
of \cite{han2023generalized}, which deforms the metric by an SPD matrix
parameter rather than by two scalars.

\textbf{Descent on the metric family.}
Write $\mathcal G$ for the metric operator defined by
$g_X^{\phi}(U,V)=\tr(\mathcal G[U]V)$. By \eqref{eq:kernel} it divides entrywise
in the eigenbasis, so its inverse multiplies entrywise,
\begin{equation}
\mathcal G^{-1}[Z]=P(K\odot Z')P^{\top},\qquad K_{ij}=\phi_{p,q}(d_i,d_j),
\label{eq:ginv}
\end{equation}
where $\odot$ is the entrywise product.
The Riemannian gradient is $\grad f=\mathcal G^{-1}[G]=P(K\odot G')P^{\top}$
for $G=\nabla f(X)$. The member $(p,q)$ enters Algorithm~\ref{alg:msd} only
through the line that forms $K$, so one implementation covers the whole family at
the cost of the eigendecomposition already incurred. For the objectives below that
order is already paid to form $G$, so the metric adds no extra cost. Where the
gradient is cheaper, the eigendecomposition sets the cost and limits the
reachable $n$.
Steps are taken with the retraction
\begin{equation}
R_X(V)=X+V+\tfrac12VX^{-1}V ,
\label{eq:retr}
\end{equation}
which stays positive definite for every symmetric $V$. 
It is a retraction in the usual sense, since $R_X(0)=X$ and
$\mathrm{D}R_X(0)[V]=V$. Most members have no exponential map in
closed form, and \eqref{eq:retr} replaces it at the price of one triple product.

\begin{algorithm}[t]
\caption{Descent on $\SPD$ in the metric \eqref{eq:metric}}
\label{alg:msd}
\begin{algorithmic}[1]
\Require $f$, $X_0\in\SPD$, member $(p,q)$, steps $t_k$
\For{$k=0,1,2,\dots$}
  \State $G\gets\nabla f(X_k)$
    \Comment{Euclidean gradient}
  \State $(d,P)\gets\operatorname{eig}(X_k)$
    \Comment{$O(n^{3})$}
  \State $K_{ij}\gets\phi_{p,q}(d_i,d_j)$
    \Comment{weight \eqref{eq:pq}}
  \State $U\gets P\bigl(K\odot(P^{\top}GP)\bigr)P^{\top}$
    \Comment{Riemannian gradient \eqref{eq:ginv}}
  \State $X_{k+1}\gets X_k-t_kU+\tfrac{t_k^{2}}{2}\,UX_k^{-1}U$
    \Comment{retract \eqref{eq:retr}}
\EndFor
\end{algorithmic}
\end{algorithm}

\section{The role of $\lcr=\lcp+\lcq$ in Hessian conditioning}
\label{sec:degree}

Finding the right $(p,q)$ for a given problem looks like a two-dimensional
search, scored by the conditioning of the Riemannian Hessian of the objective in
the new metric. Below, we give a principled way to choose $r$, $p$ and $q$.

The useful quantity is the condition number $\kappa_{\phi}$ of the Riemannian
Hessian at the solution, which governs the asymptotic rate of Riemannian gradient
descent under the usual local assumptions \cite{gao2025product}.

Let $X_\star$ be a critical point of $f$, and write its eigendecomposition as
$X_\star=P\operatorname{diag}(d_1,\dots,d_n)P^{\top}$. With $e_1,\dots,e_n$ the
standard basis of $\R^{n}$, let
\[
E_{ii}=Pe_ie_i^{\top}P^{\top},\qquad
E_{ij}=\tfrac{1}{\sqrt2}P(e_ie_j^{\top}+e_je_i^{\top})P^{\top}\ (i<j)
\]
be the frames of $\Sym$ built from the eigenbasis of $X_\star$, one per unordered
pair, so $m=n(n+1)/2$ in all. They are orthonormal for the trace inner product.
Write $\mathcal B=\nabla^{2}f(X_\star)$ for the Euclidean Hessian at $X_\star$,
taken positive definite so that condition numbers are defined, and
\begin{equation}
b_{ij}=\tr\bigl(E_{ij}\,\mathcal B[E_{ij}]\bigr)
\label{eq:bij}
\end{equation}
for its diagonal on those frames. Call $\mathcal B$ a \emph{Schur multiplier}
when it rescales each entry of the eigenbasis without mixing entries, so that
$\mathcal B[E_{ij}]=b_{ij}E_{ij}$.

\begin{proposition}[Riemannian Hessian at a critical point]
\label{prop:hess}
At a critical point $X_\star$ of $f$, in every member of \eqref{eq:pq},
\begin{equation}
\Hess f(X_\star)=\mathcal G^{-1}\circ\mathcal B ,
\label{eq:hess}
\end{equation}
which is self-adjoint for $g^{\phi}$ and so has real spectrum. Its matrix in the
$g^{\phi}$-orthonormal frames is $\mathcal B$ scaled row and column by
$\phi(d_i,d_j)^{1/2}$, so its diagonal there is
\begin{equation}
\lambda_{ij}=\phi(d_i,d_j)\,b_{ij}.
\label{eq:lambda}
\end{equation}
\end{proposition}

\begin{proof}
Since $\SPD$ is open in $\Sym$, the Levi-Civita connection of $g^{\phi}$ is the
directional derivative plus a term bilinear in its two arguments. Differentiating
$\grad f=\mathcal G^{-1}[\nabla f]$ therefore splits $\Hess f(X)[U]$ into
$\mathcal G^{-1}\bigl[\nabla^{2}f(X)[U]\bigr]$ and a remainder of two pieces, the
derivative of $\mathcal G^{-1}$ applied to $\nabla f(X)$ and the connection term
evaluated at $\grad f(X)$. Both are linear in $\nabla f(X)$. At a critical point
$\nabla f(X_\star)=0$, so the remainder vanishes and \eqref{eq:hess} follows. The
remainder carried every appearance of the derivative of the metric, which is why
only $\mathcal G^{-1}$ survives.

Self-adjointness follows because
\[
g^{\phi}\bigl(\mathcal G^{-1}[\mathcal B[U]],V\bigr)=\tr\bigl(\mathcal B[U]V\bigr)
\]
is symmetric in $U$ and $V$, since $\mathcal B$ is a Euclidean Hessian. For the
last claim, \eqref{eq:ginv} gives
$\mathcal G^{-1}[E_{ij}]=\phi(d_i,d_j)E_{ij}$, so the frames stay orthogonal under
$g^{\phi}$ but carry $g^{\phi}(E_{ij},E_{ij})=1/\phi(d_i,d_j)$, and
$\phi(d_i,d_j)^{1/2}E_{ij}$ are the $g^{\phi}$-orthonormal ones. In these orthonormal frames, \eqref{eq:hess} gives a rescaled version of the matrix of $\mathcal B$ in the frames $E_{ij}$. Each row and column indexed by $(i,j)$ is multiplied by $\phi(d_i,d_j)^{1/2}$. The diagonal entries are therefore $\phi(d_i,d_j) b_{ij}$, as in \eqref{eq:lambda}. 
\end{proof}

\begin{corollary}[Diagonal bound]
\label{cor:diagbound}
The condition number $\kappa_{\phi}$ of the Riemannian Hessian
$\Hess f(X_\star)$ obeys
\begin{equation}
\kappa_{\phi}\ \geq\ \frac{\max_{ij}\lambda_{ij}}{\min_{ij}\lambda_{ij}}
\label{eq:diagbound}
\end{equation}
for every $\mathcal B$, since a diagonal entry of a symmetric matrix is a
Rayleigh quotient. Equality holds when $\mathcal B$ is a Schur multiplier, where
the frames are eigenvectors and the $\lambda_{ij}$ are the whole spectrum.
\end{corollary}

The metric reaches the Hessian only through $\mathcal G^{-1}$ and never through
its derivative, so choosing a member of \eqref{eq:pq} is always choosing a
diagonal preconditioner in the eigenbasis of the current iterate, a valid local
choice near the critical point. Write $\kappa_{p,q}$ for $\kappa_{\phi}$ at
$\phi=\phi_{p,q}$, and $\kappa_{r}$ when that member is the diagonal one
$p=q=r/2$. Plain $\kappa=\kappa(X_\star)=d_{\max}/d_{\min}$ is the condition
number of $X_\star$ itself.
Two particular cases on $\mathcal B$ are useful to consider: (i) $\mathcal B$ has the
\emph{diagonal pure power} of degree $\gamma-2$ when $b_{ii}=c\,d_i^{\,\gamma-2}$
for every $i$, and (ii) it has the \emph{full pure power} when all entries of the frame diagonal obey $b_{ij}=c\,(d_id_j)^{(\gamma-2)/2}$. 
The
second implies the first. Define $r^{\star}=2-\gamma$ throughout. Here $c>0$ is a
scale factor shared by all frames, and it affects no condition number below,
since $\kappa_{\phi}$ is a ratio of eigenvalues.

\begin{proposition}[Conditioning floor]
\label{prop:degree}
Let $\phi$ be any weight in \eqref{eq:kernel} that is positively homogeneous of
degree $r$, and let $\mathcal B$ have the diagonal pure power. Then
\begin{equation}
\kappa_{\phi}\ \geq\
\frac{\max_{ij}\lambda_{ij}}{\min_{ij}\lambda_{ij}}\ \geq\
\kappa^{\,|r-r^{\star}|}. 
\label{eq:lower}
\end{equation}
This bound holds independently of the values of $b_{ij}$ for $i<j$ and of the couplings between distinct frames. 
\end{proposition}

\begin{proof}
The first inequality is \eqref{eq:diagbound}. For the second, only the pairs
$(i,i)$ are needed. Homogeneity gives
$\phi(d,d)=d^{r}\phi(1,1)$, so $\lambda_{ii}=c\,\phi(1,1)\,d_i^{\,r+\gamma-2}$,
whose ratio at $d_{\max}$ and at $d_{\min}$ is $\kappa^{\,r-r^{\star}}$. The
largest $\lambda_{ij}$ over the smallest is therefore at least that ratio, and at
least its reciprocal since either of the two may be the larger, hence at least
$\kappa^{\,|r-r^{\star}|}$.
\end{proof}

It should be noted that the floor (right hand side) binds every kernel metric of
degree $r$, not only members of \eqref{eq:pq}. This is true for the
Alpha-Procrustes metrics \cite{zhou2026alpha}. Their metric operator is diagonal
in the frames $E_{ij}$ and equal to $d^{2\alpha-2}$ on the $E_{ii}$ by \cite[Theorem~4] {zhou2026alpha}, so their weight carries degree $2(1-\alpha)$ and
\eqref{eq:lower} puts their floor at $\kappa^{\,|2(1-\alpha)-r^{\star}|}$, which at
$r^{\star}=0$ is the $\kappa^{\,2|\alpha-1|}$ law they report alongside \cite[Theorem~6]{zhou2026alpha}.
The paper \cite{zhou2026alpha} recommends $\alpha=1$, whose degree is
zero, so the floor it inherits is $\kappa^{\,|r^{\star}|}$. For objectives with
$r^{\star}=0$ this is $1$, and \eqref{eq:lower} leaves the member free.

\begin{corollary}[When the floor is attained]
\label{cor:attain}
Assume in addition that $\mathcal B$ is a Schur multiplier and has the full pure
power. Every $\lambda_{ij}$ is then the value at $(d_i,d_j)$ of the single
function
\[
\lambda(x,y)=c\,\phi(x,y)\,(xy)^{(\gamma-2)/2},\qquad x,y>0 .
\]
Furthermore, if $\lambda$ is monotone in each argument, then both inequalities in \eqref{eq:lower}
are equalities.
\end{corollary}

\begin{proof}
The first inequality becomes an equality by Corollary~\ref{cor:diagbound}. For
the second, a
symmetric $\lambda$ monotone in one argument is monotone the
same way in the other, so its extremes over the pairs $(d_i,d_j)$ sit where both
arguments are extreme, at $(d_{\max},d_{\max})$ and $(d_{\min},d_{\min})$.
Both extrema occur at diagonal frames $E_{ii}$. By the proof of Proposition~\ref{prop:degree}, the ratio of the maximum to the minimum is $\kappa^{\,|r-r^{\star}|}$. Thus both inequalities in \eqref{eq:lower} are equalities. 
\end{proof}

The diagonal member always meets this hypothesis. At $p=q=r/2$ the function
$\lambda$ is the pure power $c\,(xy)^{(r+\gamma-2)/2}$, hence monotone at every
$r$.

The right-hand side of \eqref{eq:lower} is minimized at $r=r^{\star}$, where it equals one. At other exponents, $\kappa_\phi$ is at least $\kappa^{\,|r-r^{\star}|}$. 
Minimizing the bound is not the same as minimizing $\kappa_{\phi}$, and the two
coincide when both inequalities become equalities.

Whether a given member meets the monotonicity hypothesis is decidable in closed
form.
Under the full pure power, $r^{\star}=2-\gamma$ gives
$(\gamma-2)/2=-r^{\star}/2$. Multiplication by this factor shifts both exponents in \eqref{eq:pq} by $-r^{\star}/2$. The function in Corollary~\ref{cor:attain} therefore becomes 
\begin{equation}
\lambda(x,y)\ \propto\ x^{\tilde p}y^{\tilde q}+x^{\tilde q}y^{\tilde p},\qquad
\tilde p=p-\tfrac{r^{\star}}{2},\quad \tilde q=q-\tfrac{r^{\star}}{2}.
\label{eq:lamab}
\end{equation}
The omitted positive factor is common to all frames and cancels from condition numbers.
The shifted
exponents carry the same two quantities as before, since
$\tilde p+\tilde q=r-r^{\star}$ and $\tilde p-\tilde q=p-q$.

\begin{proposition}[Monotonicity criterion]
\label{prop:quadrant}
The function \eqref{eq:lamab} is monotone in each argument on $(0,\infty)^{2}$
if and only if $\tilde p\,\tilde q\geq0$, equivalently
\begin{equation}
|p-q|\ \leq\ |r-r^{\star}| .
\label{eq:budget}
\end{equation}
\end{proposition}

\begin{proof}
The derivative of $\lambda$ in $x$ is
$\tilde p\,x^{\tilde p-1}y^{\tilde q}+\tilde q\,x^{\tilde q-1}y^{\tilde p}$,
nonnegative everywhere when $\tilde p,\tilde q\geq0$ and nonpositive everywhere
when $\tilde p,\tilde q\leq0$. If instead $\tilde p>0>\tilde q$ then
$\tilde p-1>\tilde q-1$, so the first term dominates as $x\to\infty$ and the
second, which is negative, dominates as $x\to0^{+}$, and the derivative changes
sign. The case $\tilde q>0>\tilde p$ is the same with the terms exchanged. The
second form follows from
$4\tilde p\,\tilde q=(r-r^{\star})^{2}-(p-q)^{2}$.
\end{proof}

How far a member may sit off the diagonal and stay monotone is therefore how far
its exponent sits from optimal. The diagonal $p=q$ meets \eqref{eq:budget} at
every exponent, which is why we fix $p=q=r/2$ and tune
$r$ alone. Monotonicity is a sufficient condition for attaining the floor, not a
necessary one, and the next proposition specifies by how much.

Under the same hypotheses, the conditioning is available in closed form, which
settles what happens off the diagonal rather than only when the floor is met.

\begin{proposition}[Exact conditioning in the pure-power Schur regime]
\label{prop:exact}
Assume $\mathcal B$ is a Schur multiplier with the full pure power, and write
$s=r-r^{\star}$ for the degree error and $a=p-q$ for the shape. Then
\begin{equation}
\kappa_{p,q}=\max\Bigl\{\kappa^{\,|s|},\ \kappa^{\,|s|/2}
\cosh\bigl(\tfrac{a}{2}\log\kappa\bigr)\Bigr\}.
\label{eq:exact}
\end{equation}
In particular $\kappa_{p,q}=\cosh(\tfrac{a}{2}\log\kappa)$ at $r=r^{\star}$, so
once $\kappa>1$ the diagonal member is the unique minimizer and every other
member misses the floor by a factor that grows with $|p-q|$.
\end{proposition}

\begin{proof}
Since $\mathcal B$ is a Schur multiplier, Corollary~\ref{cor:diagbound} makes the
$\lambda_{ij}$ the whole spectrum, so $\kappa_{p,q}$ is the ratio of the largest eigenvalue to the smallest, and Corollary~\ref{cor:attain} makes each one $\lambda(d_i,d_j)$. Up to
a positive constant, \eqref{eq:lamab} is \eqref{eq:pq} with both exponents
lowered by $r^{\star}/2$, so the factorization \eqref{eq:known_family_metric}
applies to it with $s$ in place of $r$ and $a$ in place of $p-q$,
\begin{equation}
\lambda(x,y)\ \propto\ (xy)^{s/2}\,
\cosh\Bigl(\frac{a}{2}\log\frac{x}{y}\Bigr),
\label{eq:normalform}
\end{equation}
the separation the family was built on, now read on the spectrum. What remains is
the largest and the smallest of \eqref{eq:normalform} over the frames, and a
logarithm puts both within reach, since it turns that product into a sum. Write
$\ell_i=\log d_i$. At a frame $E_{ij}$ the logarithm of \eqref{eq:normalform} is
an affine function of $(\ell_i,\ell_j)$ carrying $s$ alone, plus
$\log\cosh(a(\ell_i-\ell_j)/2)$, which is convex and nonnegative. The values
below are scaled by $(d_{\min}d_{\max})^{-s/2}$, which the ratio cancels.

\emph{Minimum.} The affine part is least over the frames where $\sqrt{d_id_j}$
is smallest if $s\geq0$ and largest if $s<0$, and either end forces $d_i=d_j$,
where the convex part vanishes. The two are least together, at
$\kappa^{-|s|/2}$.

\emph{Maximum.} The sum is convex on the square $[\ell_{\min},\ell_{\max}]^{2}$,
so it is maximal at a corner, and every corner is a frame. The two diagonal
corners are the $E_{ii}$ at $d_{\min}$ and at $d_{\max}$, with values
$\kappa^{-s/2}$ and $\kappa^{\,s/2}$, and the other two are the single frame
pairing $d_{\min}$ with $d_{\max}$, where the power in \eqref{eq:normalform}
cancels that scaling exactly and the value is $\cosh(\tfrac{a}{2}\log\kappa)$. The
maximum is therefore
\[
\max\bigl\{\kappa^{\,|s|/2},\ \cosh(\tfrac{a}{2}\log\kappa)\bigr\},
\]
and dividing it by the minimum gives \eqref{eq:exact}.
\end{proof}

Only $\kappa$ enters, not the interior of the spectrum. By \eqref{eq:exact} a
member meets the floor $\kappa^{\,|s|}$ exactly when
$\cosh(\tfrac{a}{2}\log\kappa)\leq\kappa^{\,|s|/2}$, that is when
\begin{equation}
|p-q|\ \leq\ \tfrac{2}{\log\kappa}
\operatorname{arccosh}\bigl(\kappa^{\,|s|/2}\bigr)
\ =\ |r-r^{\star}|+\tfrac{2\log2}{\log\kappa}+O\bigl(\kappa^{-|s|}\bigr).
\label{eq:exactbudget}
\end{equation}
This is wider than the monotonicity budget \eqref{eq:budget} by
$2\log2/\log\kappa$, so Proposition~\ref{prop:quadrant} is the limit of
\eqref{eq:exactbudget} as $\kappa$ grows, and the slack it leaves out is the room
a non-monotone member has to attain the floor anyway.

Inside that budget \eqref{eq:exact} returns $\kappa^{\,|s|}$ whatever the shape
is, so $\kappa_{p,q}$ is flat in $p-q$ across a band and grows like
$\kappa^{\,|p-q|/2}$ only outside it. The band closes exactly at $r=r^{\star}$,
where \eqref{eq:exactbudget} has width zero. Within the pure-power Schur regime, the diagonal member $p=q=r/2$ minimizes $\kappa_{p,q}$ at each fixed $r$. Thus it suffices to tune $r$. The same reading weighs the two parameters
against each other: moving the degree by $t$ costs $\kappa^{\,t}$, while moving
the shape by $t$ costs $\cosh(\tfrac{t}{2}\log\kappa)$, which carries half that
exponent, so the degree is worth twice the shape.

\section{SELECTION RULES FOR $\lcr$, $\lcp$ AND $\lcq$}
\label{sec:choose}

In the pure-power Schur regime, Corollary~\ref{cor:attain} settles $p$ and $q$
once $r$ is fixed, namely $p=q=r/2$. That member attains \eqref{eq:lower} at
every $r$, and at $r=r^{\star}$ it is the only member \eqref{eq:budget} admits.
It is also the cheapest member, by \eqref{eq:slice}, and it names the
exponent as $r^{\star}=2-\gamma$, so only $\gamma$ remains to be estimated.

\textbf{A Hessian that is a power congruence.}
Proposition~\ref{prop:degree} constrains the frame diagonal, so a rule must be
stated in those terms. The condition is that the Hessian act by congruence with a
power of $X$,
\begin{equation}
\nabla^{2}f(X)[U]=c\,X^{(\gamma-2)/2}\,U\,X^{(\gamma-2)/2}
\label{eq:congr}
\end{equation}
for some $c>0$, at $X=X_\star$, since that gives
$b_{ij}=c\,(d_id_j)^{(\gamma-2)/2}$ exactly, which is the full pure power, and
$r^{\star}=2-\gamma$. Fix $T$ and $C$ in $\SPD$. The Hessian of
$\tfrac12\frob{X-T}^{2}$ is $U\mapsto U$, so $\gamma=2$ and $r^{\star}=0$. The
Hessian of $\tr(CX)-\log\det X$ is $U\mapsto X^{-1}UX^{-1}$, so $\gamma=0$ and
$r^{\star}=2$. The Hessian of $\tfrac12\frob{X^{-1}-T}^{2}$ is $U\mapsto X^{-2}UX^{-2}$ at its minimizer $X_\star=T^{-1}$. Thus $\gamma=-2$ and $r^{\star}=4$, above the degrees $0$, $1$ and $2$ of the Euclidean, Bures--Wasserstein and affine-invariant metrics, respectively. The first two
meet \eqref{eq:congr} at every $X$, so they name $r^{\star}$ before $X_\star$
is known, while the third meets it only at the minimizer. 

\textbf{A Hessian of mixed degree.}
A sum of terms of different degrees, such as an evidence lower bound or a
regularized loss, meets no single \eqref{eq:congr}, so the exponent has to be
estimated rather than read off. Write $u_{ij}=\log\sqrt{d_id_j}$. The diagonal
member has $\phi(d_i,d_j)=(d_id_j)^{r/2}=e^{\,ru_{ij}}$ by \eqref{eq:slice}, so
\eqref{eq:lambda} reads $\log\lambda_{ij}=r\,u_{ij}+\log b_{ij}$, affine in $r$.
The spread of the frame diagonal in the logarithm, which is the spectrum
itself only when $\mathcal B$ is a Schur multiplier, is then a quadratic in
$r$ with a closed-form minimizer,
\begin{equation}
\hat r=\arg\min_{r}\,\sum_{i\leq j}
\bigl(\log\lambda_{ij}-\overline{\log\lambda}\bigr)^{2}
=-\,\frac{\sum_{i\leq j}(u_{ij}-\bar u)\,\log b_{ij}}
{\sum_{i\leq j}(u_{ij}-\bar u)^{2}},
\label{eq:fit}
\end{equation}
where a bar is the mean over the $m$ frames. The right-hand side is the negative of the least-squares slope of $\log b_{ij}$ against $u_{ij}$. Thus $\hat r$ can be computed in a single pass over the frame data. The estimate is well defined when all $b_{ij}>0$ and the $u_{ij}$ are not all equal. The latter condition fails precisely when $X=cI$ for some $c>0$. In the pure power case,
it recovers $r^{\star}$ exactly. From here we index the $m$ frames by
$k=1,\dots,m$ and write $u_k$ and $v_k=\log b_k$ for the pair each one carries,
so that $\log\lambda_k=r\,u_k+v_k$.

The sum of squares is a surrogate. What \eqref{eq:diagbound} depends on is the
range of $\log\lambda_{k}$ rather than its spread, and minimizing that range,
\begin{equation}
\begin{aligned}
\check r&=\arg\min_{r}\,R(r),\\[-2pt]
R(r)&=\max_{k\leq m}\bigl(r\,u_{k}+v_{k}\bigr)
-\min_{k\leq m}\bigl(r\,u_{k}+v_{k}\bigr),
\end{aligned}
\label{eq:minimax}
\end{equation}
is a one-dimensional convex problem on the same $m$ numbers, since a maximum of
affine functions is convex. Both criteria are minus the slope of $v$ against
$u$, \eqref{eq:fit} fitted in $\ell_2$ and \eqref{eq:minimax} in $\ell_{\infty}$.
The logarithm of \eqref{eq:diagbound} is a range, and
$R(r)=2\min_{c}\max_{k}\lvert r\,u_{k}+v_{k}-c\rvert$ depends on the two extreme
frames alone, where the sum of squares depends on all $m$.
Note this can be solved as a linear program \cite{boyd2004convex}.

\textbf{A sampling approach to compute $\check r$ and $\hat{r}$ efficiently.}
Both criteria read the same $m$ numbers $b_{ij}$, one per frame, a count
quadratic in $n$. A single eigendecomposition of $X$ supplies the $d_i$ and the
frames, and each $b_{ij}$ of \eqref{eq:bij} then costs one Hessian-vector
product, since it pairs $E_{ij}$ against $\nabla^{2}f(X)[E_{ij}]$, which a
directional derivative of the gradient delivers without ever forming
$\mathcal B$. Tuning is therefore one eigendecomposition and $m$ products, paid
once against a run of many iterations. A slope is a two-point quantity, though,
and the frames are far from equally informative about it, so it is worth asking
what a subset costs.

\begin{proposition}[Estimation with fewer frames]
\label{prop:subset}
Fix $X\in\SPD$ with $\kappa=d_{\max}/d_{\min}>1$ and let $u_k$, $v_k$ for
$k=1,\dots,m$ be its frame data, so that $\log\lambda_k=r\,u_k+v_k$ on the
diagonal member by \eqref{eq:lambda} and \eqref{eq:slice}. Write $S_0$ for the
pair of diagonal frames $E_{ii}$ at $d_i=d_{\min}$ and at $d_i=d_{\max}$. Since
$u_{ij}=\tfrac12(u_{ii}+u_{jj})$, every $u_k$ lies in
$[\log d_{\min},\log d_{\max}]$, and $S_0$ attains both endpoints. Let
$\hat r$ and $\check r$ minimize \eqref{eq:fit} and \eqref{eq:minimax} over all
$m$ frames, and let $\varepsilon\in\R^{m}$ be the residual of the fit,
\begin{equation}
\varepsilon_k=(v_k-\bar v)+\hat r\,(u_k-\bar u) .
\label{eq:resid}
\end{equation}
For a nonempty $S\subseteq\{1,\dots,m\}$ write $\hat r_S$ and $\check r_S$ for
the minimizers of \eqref{eq:fit} and \eqref{eq:minimax} taken over $S$ alone,
let $\varepsilon_S$ and $R_S$ be the residual and the range restricted to $S$,
and put $\sigma_S^{2}=\sum_{k\in S}(u_k-\bar u_S)^{2}$ with $\bar u_S$ the mean
of $u$ on $S$. Then the following hold.
\begin{enumerate}
\item[(i)] If $\sigma_S>0$, then
\begin{equation}
\bigl|\hat r_S-\hat r\bigr|\ \leq\ \frac{\norm{\varepsilon_S}_2}{\sigma_S}.
\label{eq:subset}
\end{equation}
\item[(ii)] If $S\supseteq S_0$, then
\begin{equation}
0\ \leq\ R(\check r_S)-R(\check r)\ \leq\ 4\norm{\varepsilon}_{\infty}.
\label{eq:twoframerange}
\end{equation}
\end{enumerate}
\end{proposition}

\begin{proof}
(i) Restricted least squares on $S$ is
\[
\hat r_S=-\frac{1}{\sigma_S^{2}}\sum_{k\in S}(u_k-\bar u_S)\,v_k .
\]
Substitute $v_k=\bar v-\hat r(u_k-\bar u)+\varepsilon_k$, which is
\eqref{eq:resid} rearranged. The two constants vanish against
$\sum_{k\in S}(u_k-\bar u_S)=0$ and the linear term returns $\hat r$, so
\[
\hat r_S=\hat r-\frac{1}{\sigma_S^{2}}\sum_{k\in S}(u_k-\bar u_S)\,\varepsilon_k ,
\]
and Cauchy--Schwarz bounds that sum by $\sigma_S\norm{\varepsilon_S}_2$.

(ii) By \eqref{eq:resid}, $\log\lambda_k=(r-\hat r)(u_k-\bar u)+\varepsilon_k$
up to a constant, which no range sees. A range of a sum is at most the sum of the
ranges, so $R(r)\leq|r-\hat r|\log\kappa+2\norm{\varepsilon}_{\infty}$, while
$R_S(r)\geq R_{S_0}(r)\geq|r-\hat r|\log\kappa-2\norm{\varepsilon}_{\infty}$
because $u$ spans $\log\kappa$ already on $S_0$ and enlarging a set only widens
its range. Hence $R-R_S\leq4\norm{\varepsilon}_{\infty}$ at every $r$. Dropping
frames lowers a maximum and raises a minimum, so $R_S\leq R$ pointwise and
$R(\check r)\geq R_S(\check r)\geq R_S(\check r_S)$. Subtracting that from
$R(\check r_S)$ bounds it by $R(\check r_S)-R_S(\check r_S)$ and so by
$4\norm{\varepsilon}_{\infty}$, and $\check r$ minimizes $R$, which gives the
first inequality.
\end{proof}

Note that on the extreme pair alone $\sigma_{S_0}^{2}=\tfrac12\log^{2}\kappa$ and
$\norm{\varepsilon_{S_0}}_2\leq\sqrt2\,\norm{\varepsilon}_{\infty}$, so
\eqref{eq:subset} leads to
\begin{equation}
\bigl|\hat r_{S_0}-\hat r\bigr|\,\log\kappa\ \leq\ 2\norm{\varepsilon}_{\infty}.
\label{eq:twoframe}
\end{equation}

\textbf{An objective built from a distance.}
For a squared-distance objective, we use the metric that defines the distance. The Riemannian Hessian of $\tfrac12\dist^{2}(\cdot,T)$ at $X_\star=T$
is the identity in the metric that defines $\dist$, so that metric is exactly
optimal for it. We call this \emph{reciprocity}. 
A least-squares fit $\tfrac12\frob{X-T}^{2}$ corresponds to the Euclidean metric,
and $\tfrac12\dist_{\mathrm{BW}}^{2}(\cdot,T)$ to Bures--Wasserstein, which the
fitted exponent cannot reach because it sits off the diagonal $p=q$. Likewise
$\tfrac12\dist_{\mathrm{AI}}^{2}(\cdot,T)$ corresponds to affine-invariant
geometry, and $\tfrac12\frob{\log X-\log T}^{2}$ to log-Euclidean, a weight the
family does not contain at all. A Fr\'echet mean averages several such terms. Under the
Euclidean and log-Euclidean metrics the Hessian is still the identity at the
barycenter, so the match stays exact, while under the others the minimizer is
none of the $T_i$ and the match is only as close as the spread of the data
allows.
This observation is consistent with \eqref{eq:congr}. By \eqref{eq:slice}, that condition makes the Euclidean Hessian a positive multiple of the metric operator at $p=q=(2-\gamma)/2$. It therefore gives the same local Hessian matching within the diagonal family. 

\section{EXPERIMENTS}
\label{sec:exp}
The experiments below score the tuned member against the named metrics a
practitioner would otherwise pick.

Every run below uses Algorithm~\ref{alg:msd} with the retraction \eqref{eq:retr}
and an Armijo line search \cite{absil2008optimization,boumal2023introduction},
starts from the same $X_0$ and stops at the same tolerance, so only
the weight differs. The tolerance is a relative objective gap.
Conditioning is always $\kappa_{\phi}$ at a common reference
solution.

\subsection{The exponent read off the Hessian}
\label{sec:congr}

We minimize the three objectives of Section~\ref{sec:choose}, the least-squares
fit, the Gaussian precision estimate and the inverse fit, whose power-congruence
Hessians name $r^{\star}=0$, $2$ and $4$ before anything is run. Here $T=C$ is a real
covariance, the class-conditional pixel covariance of a standard digits set at
$n=36$, and all three problems share $\kappa(X_\star)=9836$ because
$\kappa(A)=\kappa(A^{-1})$, so the objective alone moves $r^{\star}$. A power
congruence is a full pure power and a Schur multiplier, so
Proposition~\ref{prop:exact} gives $\kappa_{p,q}$ in closed form, and for a
diagonal weight of degree $r$ it reduces to $\kappa^{\,|r-r^{\star}|}$, which
Table~\ref{tab:congr} confirms to the digit. Off the diagonal a weight of the
right degree need not attain the floor, and the table shows both outcomes.
Bures--Wasserstein does attain it, because the extreme
$\lambda_{ij}$ fall on the frames $E_{ii}$ where its weight is a multiple of the
diagonal one of the same degree. It does so in all three rows only because no
$r^{\star}$ among them lies near its own degree. Read at $(p,q)=(1,0)$,
\eqref{eq:exact} meets the floor exactly when
$|1-r^{\star}|\geq\tfrac{2}{\log\kappa}\log\cosh(\tfrac12\log\kappa)$, which is
$0.85$ here and rises toward $1$ as $\kappa$ grows, against the $1$, $1$ and $3$
the three rows supply. At $r^{\star}=1$ its degree is exactly right and it still
pays $\cosh(\tfrac12\log\kappa)=49.6$ where the diagonal member $p=q=\tfrac12$
pays $1$, a penalty of order $\sqrt{\kappa}/2$. Matching the degree is therefore
never enough on its own.
Log-Euclidean is the other outcome, carrying the right degree on the second
objective
and still missing the floor by $116$ because it is built from the logarithmic
rather than the geometric mean, so the degree is necessary and not sufficient.
The third row is the one to note. A standard estimator has $r^{\star}=4$,
two degrees past affine-invariant, and there the best named weight is
$9.7\times10^{7}$ worse conditioned than the diagonal member at $r^{\star}$.
The frame diagonal is a pure power here, so $v_k$ is exactly affine in $u_k$,
the residual of Proposition~\ref{prop:subset} vanishes, and its pair $S_0$
suffices: two
Hessian-vector products in place of $m=666$ return each of $r^{\star}=0$, $2$ and
$4$ to twelve digits.

\begin{table}[t]
\centering
\setlength{\tabcolsep}{2.5pt}
\caption{The congruence rule on a real covariance, $n=36$ and
$\kappa(X_\star)=9836$. Entries are $\log_{10}\kappa_{\phi}$, so \eqref{eq:lower}
predicts $3.99\,|r-r^{\star}|$ from the degree $r$ of the weight alone. The three
objectives are those of Section~\ref{sec:choose} in order, and AP is the
Alpha-Procrustes member at $\alpha=1$. The tuned members in the first two rows are Euclidean ($r^{\star}=0$) and affine-invariant ($r^{\star}=2$), respectively. } 
\label{tab:congr}
\begin{tabular}{@{}lc cccccc@{}}
\toprule
objective & $r^{\star}$ & Eucl. & BW & aff.-inv. & log-Eucl. & AP & tuned \\
\midrule
least squares & $0$ & $\mathbf{0}$ & $3.99$ & $7.99$ & $7.99$ & $0.30$
& $\mathbf{0}$ \\
precision & $2$ & $7.99$ & $3.99$ & $\mathbf{0}$ & $2.07$ & $7.99$
& $\mathbf{0}$ \\
inverse fit & $4$ & $15.97$ & $11.98$ & $7.99$ & $7.99$ & $15.97$
& $\mathbf{0}$ \\
\bottomrule
\end{tabular}
\end{table}

\subsection{An exponent past every named metric}
\label{sec:ml}

\begin{figure}[t]
\centering
\includegraphics[width=\columnwidth]{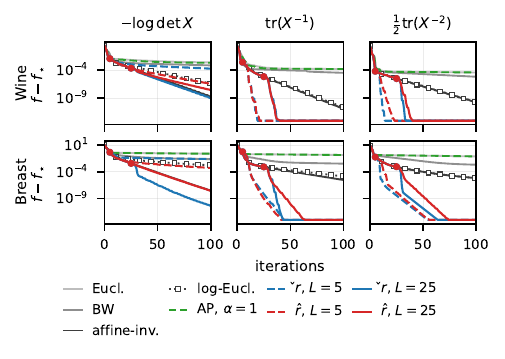}
\caption{Objective gap for \eqref{eq:ml} against iterations at $\mu=0.01$. Rows correspond to the two datasets, and columns to the three regularizers. The regularizers are labeled by $R-\tr X$, as in Table~\ref{tab:mu}. All runs start from $X_0=I$ and use the same line search. The two-stage curves run at the default
$r=2$, the affine-invariant weight, until the marker at $L$ and branch
there. A five-step pilot suffices on the barriers but not under the
log-determinant, which is why $L=25$ is used throughout. Log-Euclidean carries
open squares because it can sit under affine-invariant to within half a percent.}
\label{fig:conv}
\end{figure}

The objective is a Mahalanobis matrix learned from labeled triplets under a
regularizer \cite{weinberger2009distance,davis2007information},
\begin{equation}
f(X)=\frac{1}{|\mathcal T|}\sum_{t\in\mathcal T}
s\bigl(1+\langle X,D_t\rangle\bigr)+\mu R(X),
\label{eq:ml}
\end{equation}
where $s(z)=\log(1+e^{z})$ is the softplus and
$D_t=(x_a-x_{+})(x_a-x_{+})^{\top}-(x_a-x_{-})(x_a-x_{-})^{\top}$ for a triplet
of an anchor $x_a$, a positive $x_{+}$ and a negative $x_{-}$. We use the Wine
and Breast Cancer Wisconsin datasets from the UCI repository \cite{dua2019uci},
as distributed with scikit-learn \cite{pedregosa2011scikit}, with a stratified
$70/30$ split repeated
three times, $|\mathcal T|=2000$ triplets per split, $\mu=0.01$ and $X_0=I$.
Features are centered. We consider three regularizers: $R=\tr X-\log\det X$, $\tr X+\tr(X^{-1})$, and $\tr X+\tfrac12\tr(X^{-2})$. Their diagonal frame curvatures are $d_i^{-2}$, $2d_i^{-3}$ and $3d_i^{-4}$, respectively. For the regularizers alone, these correspond to $r^{\star}=2$, $3$ and $4$. We call the last two
barriers. The log-determinant is the usual choice on $\SPD$, and the exponent it
sets is exactly affine-invariant, so nothing is left to tune there. The two
barriers grow faster at the boundary and carry $r^{\star}$ past every named
metric, which is what makes them worth running.
The linear term adds nothing to the Hessian and is there only to bound
$X_\star$ above. The triplet loss carries no degree at all, so the sum carries
none, no rule of
Section~\ref{sec:choose} names the exponent and it has to be estimated, from the
least-squares fit \eqref{eq:fit} or the range criterion \eqref{eq:minimax}. Both
read the frame diagonal $b_{ij}$ and the eigenvalues $d_i$ at a point, and the
point that matters is the solution being sought, so we run
Algorithm~\ref{alg:msd} in two stages. Descend $L$ steps from $X_0$ at the
default exponent $r=2$, evaluate the criterion at the point reached, then
continue from that point at the exponent it returns. Every iteration count we
report includes the $L$ pilot steps. Reading the frame diagonal costs $m$
Hessian-vector products in general, but they batch into one pass over the
triplets here, so either criterion costs $0.9$ to $1.5$ iterations.

Figure~\ref{fig:conv} plots the objective gap against iterations for all six
settings. On the four barrier panels both criteria reach machine precision inside
the budget and no named weight does, because the exponent they return sits near
$3$ or $4$ and the nearest named degree is $2$. The range criterion returns
$3.45$ and $4.15$ on Wine and $3.10$ and $4.05$ on Breast under the two barriers,
where the least-squares fit runs higher, to $4.87$. Under the log-determinant
both return near $2$, and that column is where the criteria add least, since
affine-invariant already carries that regularizer's exponent.

Figure~\ref{fig:budget} varies the number of frames used for fitting from two to $m$. Each subset contains the pair $S_0$ from Proposition~\ref{prop:subset}. The remaining frames are sampled at random. Dropping frames costs little, and for one of the two
criteria it gains. The range criterion is flat, so the extra frames are not
required. The least-squares fit, on the other hand, is better off with fewer.
No subset drawn violated \eqref{eq:subset}, \eqref{eq:twoframerange} or
\eqref{eq:twoframe}.
Read with Table~\ref{tab:congr}, where the residual vanishes so that any two
Hessian-vector products reproduce every $r^{\star}$ exactly, the $n$ diagonal
frames suffice in both regimes.

\begin{figure}[t]
\centering
\includegraphics[width=\columnwidth]{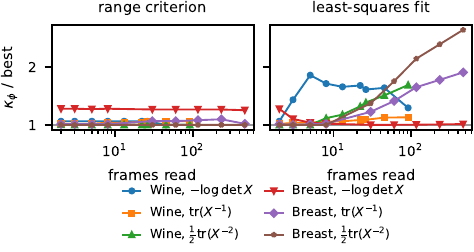}
\caption{Conditioning delivered against the number of frames read, over the six
settings of \eqref{eq:ml} at $\mu=0.01$, with the range criterion
\eqref{eq:minimax} on the left and the least-squares fit \eqref{eq:fit} on the
right. Curves are named by dataset and by $R-\tr X$ as in Table~\ref{tab:mu}.
Each value is $\kappa_{\phi}$ at the exponent the criterion returns, divided by
the best any diagonal member attains, so $1$ is the floor of the sweep. Each subset contains the pair $S_0$ from Proposition~\ref{prop:subset}. Additional frames are sampled at random. The curves show the median over twelve draws at each subset size. Reading more frames does nothing for the range criterion and hurts the least-squares fit. }
\label{fig:budget}
\end{figure}

Table~\ref{tab:mu} sweeps the regularization weight, since $\mu=0.01$ is a single
setting and the advantage need not survive elsewhere. Each entry is
a gain, the condition number of the best named metric divided by the one the
criterion delivers, so a value above one is the factor by which tuning improves
the conditioning and a value below one is a loss. The two barriers gain in every
setting, and by the widest margin where the regularization is weakest. The
log-determinant is the exception, gaining little and sometimes losing, since
affine-invariant already carries the exponent it sets.

\textbf{On Alpha-Procrustes.} The weights of \cite{zhou2026alpha} carry
degree $2(1-\alpha)$, so the member at $\alpha=1$ they recommend has degree
zero, like the Euclidean case. 
Its weights are given by $2(x^{2}+y^{2})/(x+y)^{2}$. This equals $1$ at $x=y$ and approaches $2$ as $x/y$ tends to $0$ or $\infty$. It agrees with the Euclidean weight on the diagonal frames $E_{ii}$, but differs on off-diagonal frames with $d_i\neq d_j$. Table~\ref{tab:congr} compares the resulting Hessian condition numbers. For the precision and inverse-fit objectives, the extreme Hessian eigenvalues occur on the diagonal frames, so the two metrics have the same condition number. For the least-squares objective, $\mathcal B$ is the identity. Euclidean then has condition number $1$, while the variation of the Alpha-Procrustes weight gives a condition number (close to) $2$. Figure~\ref{fig:conv} also shows similar convergence curves for these two metrics from $X_0=I$.

\begin{table}[t]
\centering
\setlength{\tabcolsep}{4pt}
\caption{Conditioning gains from tuning the exponent for different regularization weights. Each entry is  $\bigl(\min_{\phi}\kappa_{\phi}\bigr)/\kappa_{r}$, where the minimum is taken over the Euclidean, Bures--Wasserstein and affine-invariant metrics. We select $r$ after the $L=25$ pilot using either the range criterion $\check r$ in \eqref{eq:minimax} or the least-squares fit $\hat r$ in \eqref{eq:fit}. Both criteria use all $m$ frames. A value above one gives the factor by which tuning improves the conditioning. The two criteria agree closely at $\mu=1$ and separate as the regularizer weakens, with $\check r$ generally ahead. }
\label{tab:mu}
\begin{tabular}{@{}l cc @{\quad} cc @{\quad} cc@{}}
\toprule
& \multicolumn{2}{c}{$\mu=1$} & \multicolumn{2}{c}{$\mu=0.1$}
& \multicolumn{2}{c}{$\mu=0.01$} \\
\cmidrule(lr){2-3}\cmidrule(lr){4-5}\cmidrule(lr){6-7}
$R-\tr X$ & $\check r$ & $\hat r$ & $\check r$ & $\hat r$ & $\check r$ & $\hat r$ \\
\midrule
\multicolumn{7}{@{}l}{\emph{Wine}} \\
$-\log\det X$ & $1.06$ & $1.02$ & $1.67$ & $1.19$ & $1.00$ & $0.74$ \\
$\tr(X^{-1})$ & $1.68$ & $1.67$ & $4.20$ & $3.90$ & $10.8$ & $10.2$ \\
$\tfrac12\tr(X^{-2})$ & $1.96$ & $1.92$ & $5.54$ & $4.68$ & $17.1$ & $11.4$ \\
\midrule
\multicolumn{7}{@{}l}{\emph{Breast}} \\
$-\log\det X$ & $1.04$ & $1.02$ & $1.05$ & $1.04$ & $0.92$ & $0.98$ \\
$\tr(X^{-1})$ & $2.63$ & $2.36$ & $6.43$ & $4.38$ & $14.3$ & $7.32$ \\
$\tfrac12\tr(X^{-2})$ & $3.88$ & $3.29$ & $10.9$ & $6.55$ & $31.1$ & $11.8$ \\
\bottomrule
\end{tabular}
\end{table}

\subsection{Tuning the shape}
\label{sec:shape-tuning}

We learn an SPD task covariance for the seven torque outputs in the
SARCOS inverse-dynamics dataset~\cite{vijayakumar2000locally}.
A multi-output Gaussian process uses this matrix to model dependence
between the outputs \cite{bonilla2007multi}.
Let $Y\in\R^{N\times7}$ contain the standardized torques and let
$\mathbf y=\operatorname{vec}(Y)$.
For an input covariance matrix $\Gamma\in\R^{N\times N}$ and noise
variance $\nu$, the model has
\begin{equation}
 \Sigma(X)=X\otimes\Gamma+\nu I_{7N},\qquad
 X\in\mathbb S_{++}^{7}.
 \label{eq:sarcos-covariance}
\end{equation}
We minimize the Gaussian negative log marginal likelihood
\cite{williams2006gaussian},
\begin{equation}
 f(X)=\frac{1}{2N}\left[
 \log\det\Sigma(X)+\mathbf y^\top\Sigma(X)^{-1}\mathbf y\right].
 \label{eq:sarcos-likelihood}
\end{equation}
The input covariance $\Gamma$ uses a radial basis function (RBF) kernel.
We calibrate its parameters and $\nu$ once at $X=I$, then hold them
fixed across methods.
We use seven training subsets, four with $N=512$, two with $N=2{,}048$
and one with $N=4{,}096$.
We standardize each training subset separately.

As in Section~\ref{sec:ml}, we descend $L=25$ steps from $X_0=I$ at
the affine-invariant exponent $r=2$, then tune at the point reached $X_L$.
The preceding experiments fit $r$ on the diagonal $p=q$.
Here we also tune the shape $a=p-q$ of Proposition~\ref{prop:exact}.
Exchanging $p$ and $q$ leaves the metric unchanged, so we work with
its magnitude $\delta=|a|$.

The separation of exponent and shape suggests a tuning heuristic based on the frame data at $X_L$. After the first exponent fit, we use the shape to raise the smaller entries of the scaled diagonal toward their maximum. We then refit the exponent using the adjusted diagonal. 
At $X_L$, read the frame diagonal $b_{ij}$ and the eigenvalues
$d_i$ as in Section~\ref{sec:choose}. All seven pilot points have positive frame curvatures and pass the numerical non-degeneracy checks for exponent fitting. 
Keep the frame data $u_k,v_k$, with $b_k=b_{ij}$ for the frame
associated with $(i,j)$. Write $\xi_k=\tfrac12\bigl|\log(d_i/d_j)\bigr|$
for that pair and evaluate the scaled diagonal \eqref{eq:lambda} at $X_L$.
The factorization \eqref{eq:known_family_metric} gives, after removing
the common factor $4^{\delta^2}$,
\begin{equation}
  \widetilde{\lambda}_k(r,\delta)=4^{-\delta^2}\lambda_k(r,\delta)=e^{r u_k+v_k}\cosh(\delta\xi_k).
 \label{eq:shape-curvatures}
\end{equation}
The factor cancels from the ratio of the largest to the smallest entry.

At $\delta=0$, let $r_0$ be $\hat r$ from the least-squares fit
\eqref{eq:fit} or $\check r$ from the range criterion \eqref{eq:minimax}.
Set $W_0=\max_k\widetilde{\lambda}_k(r_0,0)$.
Choose the largest $\delta\geq0$ for which every
$\widetilde{\lambda}_k(r_0,\delta)$ is at most $W_0$.
For a nonscalar spectrum, this gives
\begin{equation}
 \delta^\star=
 \min_{\xi_k>0}
 \frac{\operatorname{arccosh}\bigl(W_0/\widetilde{\lambda}_k(r_0,0)\bigr)}{\xi_k}.
 \label{eq:shape-choice}
\end{equation}
Each frame with $\xi_k>0$ gives an upper bound on $\delta$.
For $0\leq\delta\leq\delta^\star$, the maximum normalized entry stays at
$W_0$ and the minimum is nondecreasing. Thus \eqref{eq:shape-choice}
minimizes their ratio over this interval.

To obtain $r_1$, repeat the fitting rule from the first step with
$\delta=\delta^\star$ in \eqref{eq:shape-curvatures}.
The selected metric has
\begin{equation}
 p=\frac{r_1+\delta^\star}{2},\qquad
 q=\frac{r_1-\delta^\star}{2}.
 \label{eq:shape-parameters}
\end{equation}
We select these parameters once and continue descent from $X_L$.
In Figure~\ref{fig:sarcos-shape}, Shape-LS uses the least-squares fit
and Shape-Range uses the range criterion.
Figure~\ref{fig:shape-sandwich} shows these three steps in the parameter plane.
Geometrically, this three-step construction can reach any member of the
$(p,q)$ family, up to exchanging $p$ and $q$.

\begin{figure}[t]
 \centering
 \includegraphics[width=0.6\columnwidth]{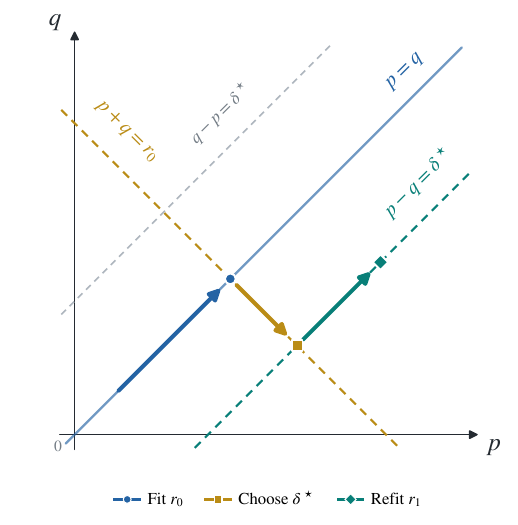}
 \caption{Selecting the exponent and shape in the $(p,q)$ family.
 Blue marks the first exponent fit on $p=q$.
 Yellow selects $\delta^\star$ along $p+q=r_0$, and green refits
 the exponent along $p-q=\delta^\star$.
 The gray branch contains the same metrics with $p$ and $q$ exchanged.
 The positions are schematic. The arrows give the selection order.
 The final fit can move in either direction along the green line.}
 \label{fig:shape-sandwich}
\end{figure}

We use the diagonal members returned by the same fitting rules as the
$r$-only baselines. We also compare with Euclidean, BW, affine-invariant
and log-Euclidean.
All eight methods continue from $X_L$ with the same Armijo history.
The four fitted methods read all $m=7(7+1)/2=28$ entries $b_{ij}$
of the frame diagonal.
Runs stop when
$[f(X)-f(X_{\rm ref})]/[f(I)-f(X_{\rm ref})]\leq10^{-8}$,
with a total budget of 2,000 iterations.
Here $X_{\rm ref}$ is a common numerical reference solution.
We evaluate $\kappa_\phi$ from the full preconditioned Hessian at $X_{\rm ref}$.

\begin{figure*}[t]
 \centering
 \includegraphics[width=0.94\textwidth]{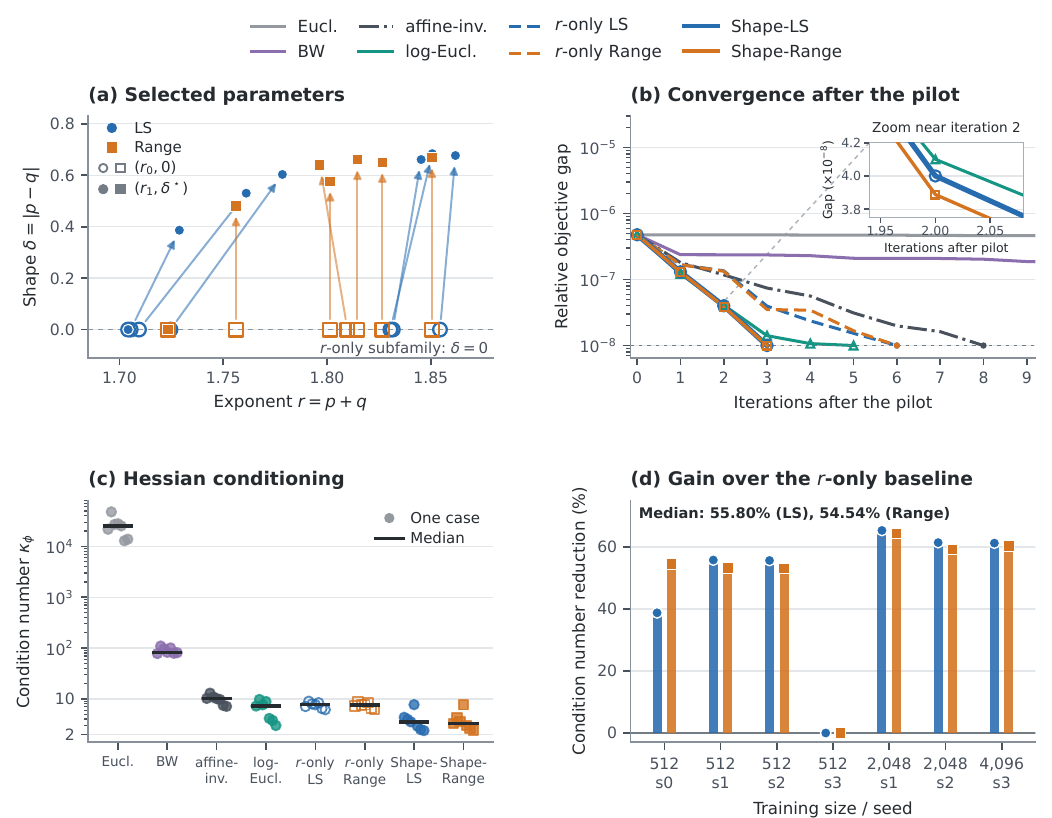}
 \caption{Exponent and shape tuning on SARCOS over seven cases.
 (a) Open markers show $(r_0,0)$ and filled markers show $(r_1,\delta^\star)$.
 (b) Median relative objective gap after the shared 25-iteration affine-invariant pilot.
 Completed runs are held at $10^{-8}$ for aggregation.
 The inset enlarges the second iteration after the pilot.
 (c) Condition number $\kappa_\phi$ of the full preconditioned Hessian at $X_{\rm ref}$.
 Each point is a case and black bars mark medians.
 (d) Percentage reduction $100(1-\kappa_\phi/\kappa_{r_0})$ relative to
 the corresponding $r$-only baseline.
 Each baseline uses the initial exponent $r_0$ from the same fitting rule
 with $\delta=0$.}
 \label{fig:sarcos-shape}
\end{figure*}

With either fitting rule, the selected shape is positive in six cases
(Fig.~\ref{fig:sarcos-shape}(a)).
The remaining case, $N=512$ with seed 3, gives $\delta^\star=0$.
The full condition number improves over the corresponding diagonal
baseline in the six positive-shape cases and stays unchanged in the seventh.
Median $\kappa_\phi$ falls from $7.70$ to $3.53$ for the least-squares fit
and from $7.64$ to $3.28$ for the range criterion
(Fig.~\ref{fig:sarcos-shape}(c)).
Relative to the corresponding diagonal member, the median reduction is
$55.80\%$ for the least-squares fit and $54.54\%$ for the range criterion
(Fig.~\ref{fig:sarcos-shape}(d)).
Both choices need a median of 28 iterations, including the pilot.
Each diagonal member needs 31, log-Euclidean needs 30 and
affine-invariant needs 33.

Both fitting rules read the diagonal $b_{ij}$ of the Euclidean Hessian
in the frames $E_{ij}$.
At a critical point where $\mathcal B$ is a Schur multiplier,
the entries $\lambda_{ij}$ give the whole Riemannian Hessian spectrum by
Corollary~\ref{cor:diagbound}.
Coupling between frames makes the diagonal ratio a surrogate for
$\kappa_\phi$.
The diagonal ratio is non-increasing under the range refit. The least-squares refit minimizes the variance of the log diagonal. 

Shape tuning reuses $d_i$, the frames and the diagonal $b_{ij}$ from
the first exponent fit. It adds one exponent fit and an $O(m)$ pass for
\eqref{eq:shape-choice}. No additional Hessian-vector products are needed.
Both least-squares fits are closed form, so the arithmetic after reading
$b_{ij}$ remains $O(m)$. For the range criterion, we solve two linear programs
instead of one. The fitting data still require $O(m)$ storage.

\section{CONCLUSION}

The family of metrics defined by $X^{p}LX^{q}+X^{q}LX^{p}=U$ contains Euclidean, Bures--Wasserstein
and affine-invariant geometry as exact members and reaches past all three. Each iteration of Algorithm~\ref{alg:msd} uses an eigendecomposition and the retraction \eqref{eq:retr}, which preserves positive definiteness. The conditioning of the Riemannian Hessian at the solution
obeys a floor set by $r=p+q$ alone, and in the pure-power Schur regime the
diagonal member $p=q=r/2$ attains that floor, so the two-parameter choice
collapses to one number. We give principled ways to estimate $r$, and show where
tuning it helps and where a named metric already suffices.  The SARCOS example also shows a further gain from tuning the shape. 


\bibliographystyle{IEEEbib}
\bibliography{references.bib}

\end{document}